\documentclass[12pt]{article}
\usepackage[centertags]{amsmath}
\usepackage{amsfonts}
\usepackage{amssymb}
\usepackage{amsthm}
\usepackage{xcolor}
\input{amssym.def}
\usepackage{graphicx}
\usepackage{lineno}
\numberwithin{equation}{section}

\newtheorem{Definition}{Definition}[section]
\newtheorem{theorem}[Definition]{Theorem}
\newtheorem{lemma}[Definition]{Lemma}
\newtheorem{proposition}[Definition]{Proposition}
\newtheorem{corollary}[Definition]{Corollary}
\newtheorem{example}[Definition]{Example}

\begin{document}

\title{\Large \bf On m-quasi-ideals in m-regular ordered semigroups}
\author{Susmita Mallick \\
\footnotesize{Department of Mathematics, Visva-Bharati
University,}\\
\footnotesize{Santiniketan, Bolpur - 731235, West Bengal, India}\\
\footnotesize{mallick.susmita11@gmail.com}}

\date{}
\maketitle

\begin{abstract}
In this paper we characterize left(right) ideals, bi-ideals and
quasi-ideals of an ordered semigroup by an index $m$ and give some
important interplays between these ideals. The concept of
m-regularity of an ordered semigroups has been introduced. Moreover
m-regular ordered semigroups are characterized by their
m-quasi-ideals and the fact that for any m-regular ordered
semigroups $A$, the set $Q_{A}$ of all m-quasi-ideals of $A$, with
multiplication defined by: $Q_{1}\circ Q_{2}=(Q_{1}Q_{2}]$, for all
$Q_{1},Q_{2}\in Q_{A}$, is a m-regular semigroup is obtained here.
\end{abstract}

\section{Introduction and preliminaries}
An ordered semigroup is a partially ordered set $(A,\leq)$, and at
the same time a semigroup $(A, \cdot)$ such that for all $a , b ,
x\in A; \;a \leq b$ implies $\;xa\leq xb \;and \;a x \leq b x$. It
is denoted by $(A,\cdot, \leq)$. Throughout this paper we consider
an ordered semigroup $A$ with $\{1\}$. For an ordered semigroup $A$
and $H \subseteq S$, denote $(H]_A := \{t \in A : t \leq  h, \;$for
some $\;h\in H\}.$ Also it is denoted by $(H]$ if there is no scope
of confusion. Let $I$ be a nonempty subset of an ordered semigroup
$A$. $I$ is a left (right) ideal \cite{K1} of $S$, if $AI \subseteq
I \;( IA \subseteq I)$ and $(I]= I$ and $I$ is an ideal of $A$ if it
is both a left and a right ideal of $A$. $A$ is left (right) simple
\cite{H1} if it has no non-trivial proper left (right) ideal. For an
ordered semigroup $A$ and a subsemigroup $S$ of $A$,
$S^{m}=SSSS\cdot\cdot\cdot\cdot S$(m-times) where $m$ is a positive
integer. Its easy to check that for any subsemigroup $S$ of an
ordered semigroup $A$, $S^{n}\subseteq S$ for all positive integer
n. Similarly $S^{r}\subseteq S^{t}$, for all positive integers $r$
and $t$, such that $r\geq t$ but the converse is not true.
\begin{lemma}\cite{kehayopulu2} \label{AB}
Let $T$ be an ordered semigroup and $A$ and $B$ be subsets of $T$.
Then the following statements hold:

\begin{enumerate}
\item \vspace{-.4cm} $A\subseteq (A]$.
\item \vspace{-.4cm} $((A]]=(A]$.
\item \vspace{-.4cm} If $A\subseteq B$ then $(A]\subseteq (B]$.
\item \vspace{-.4cm} $(A\cap B]\subseteq (A]\cap (B]$.
\item \vspace{-.4cm} $(A\cup B]=(A]\cup (B]$.
\item \vspace{-.4cm} $(A](B]\subseteq (AB]$.
\item \vspace{-.4cm} $((A](B]]=(AB]$.
 \end{enumerate}
\end{lemma}
\section{\textbf{m-left/right ideal }}

\begin{Definition}
Let $A$ be an ordered semigroup. A subsemigroup $L$ of $A$ is called
m-left ideal of $A$ if $A^{m}L\subseteq L$ and $(L]=L$, where $m$ is
a positive integer not necessarily $1$.
\end{Definition}
Similarly, a subsemigroup $R$ of $A$ is called m-right ideal of $A$
if $RA^{m}\subseteq R$ and $(R]=R$, where m is a positive integer. A
subsemigroup $I$ of $A$ is called an m-two sided ideal or simply an
m-ideal of $A$ if it is both m-left ideal and m-right ideal of $A$.

\begin{proposition} \label{1}
Let $A$ be an ordered semigroup. Then following assertions hold:
\begin{enumerate}
  \item \vspace{-.4cm}
Every left(right) ideal is an m-left(right)ideal of $A$.
  \item \vspace{-.4cm}
Every m-left(right) ideal is an n-left(right) ideal of $A$, for all
$n\geq m$.
 \item \vspace{-.4cm}
Intersection of m-left ideals of $A$ (If non-empty) is an m-left
ideal of $A$.
 \item \vspace{-.4cm}
Intersection of m-right ideals of $A$ (If non-empty) is an m-right
ideal of $A$.
\end{enumerate}
\end{proposition}

\begin{proof}

$(1)$: Let $L$ be a left ideal of $A$ then $A^{m}L\subseteq
AL\subseteq L$ and $(L]=L$. Hence $L$ is a m-left ideal of $A$.

Similarly every right ideal of $A$ is an m-right ideal of $A$.

$(2)$: First consider $L$ be an m-left ideal of $A$, then
$A^{m}L\subseteq L$ and $(L]=L$. Now for $n\geq m$; $A^{n}L\subseteq
A^{m}L\subseteq L$ and $(L]=L$. Hence $L$ is an n-left ideal of $A$.

Likewise, every m-right ideal of $A$ is an n-right ideal of $A$.

$(3)$: Let $\{L_{\lambda}: \lambda\in \Lambda \}$ be a family of
m-left ideals of an ordered semigroup $A$. Let $L=\cap_{\lambda\in
\Lambda} L_{\lambda}$ is a subsemigroup of $A$ being intersection of
a family of subsemigroups of $A$. Now we have,
$A^{m_{\lambda}}L_{\lambda}\subseteq L_{\lambda}$, for all
$\lambda\in \Lambda$ and $(L_{\lambda}]=L_{\lambda}$. Now
$L\subseteq L_{\lambda}$, for all $\lambda\in \Lambda$. Let
$m=max\{m_{\lambda}: \lambda\in \Lambda\}$. Hence $A^{m}L\subseteq
A^{m_{\lambda}}L\subseteq A^{m_{\lambda}}L_{\lambda}\subseteq
L_{\lambda}$, for all $\lambda\in \Lambda$. Therefore
$A^{m}L\subseteq \cap_{\lambda\in \Lambda}L_{\lambda}=L$ implies
$A^{m}L\subseteq L$. Now $(L]=(\cap_{\lambda\in
\Lambda}L_{\lambda}]\subseteq \cap_{\lambda\in \Lambda}
(L_{\lambda}]=\cap_{\lambda\in \Lambda} L_{\lambda}=L$ implies
$(L]=L$. Hence $L$ is a m-left ideal of $A$.

$(4)$:

Analogously.
\end{proof}

The converse of the above statement is incorrect. This is verified
by the examples.

\begin{example}\label{example1}
Let $A=\{x,y,z,w\}$ be an ordered semigroup with the multiplication
'$\cdot$' and the order relation defined by

\begin{tabular}{|c|c|c|c|c|c|}
  \hline

  . & x & y & z & w  \\ \hline
  x & w & z & w & w \\ \hline
  y & z & w & w & w  \\ \hline
  z & w & w & w & w  \\ \hline
  w & w & w & w & w \\

  \hline

\end{tabular}\\

$\leq =\{(x,x),(y,y),(z,z),(w,x),(x,y),(x,z),(w,w)\}$.

Let $L=\{x,w\}$. For integer $m>1$, we obtain that $L$ is an m-left
ideal ,m-right ideal of $A$ but not a left(right) ideal of $A$.

\end{example}

Let $K$ be a subsemigroup of an ordered semigroup $A$ and
$\mathcal{L}=\{L: L \ is \ an\ m-left\ ideal\ of\ A$ containing
$K\}$. Therefore $\mathcal{L}$ is a non-empty because $A\in
\mathcal{L}$. Let $(K)_{m-l}=\cap_{L\in \mathcal{L}}L$. Hence by
Theorem \ref{1} $(K)_{m-l}$ is an m-left ideal of an ordered
semigroup $A$. Moreover we can easily check that $(K)_{m-l}$ is the
smallest m-left ideal of $A$ containing $K$. The m-left ideal
$(K)_{m-l}$ is called principal m-left ideal of $A$ generated by
$K$.

\begin{theorem}\label{Principal left ideal}
Let $K$ be a subsemigroup of an ordered semigroup $A$. Then
\begin{enumerate}
 \item \vspace{-.4cm}
m-left ideal generated by $K$ is defined by $(K)_{m-l}=(K\cup
A^{m}K]$.
  \item \vspace{-.4cm}
m-right ideal generated by $K$ is defined by $(K)_{m-r}=(K\cup
KA^{m}]$.
\end{enumerate}
\end{theorem}
\begin{proof}
$(1)$: Suppose $(K)_{m-l}=(K\cup A^{m}K]$. We must explain that
$(K)_{m-l}$ is the minimal m-left ideal of $A$ which contains $K$.
Now $(K\cup A^{m}K](K\cup A^{m}K]\subseteq (KK\cup KA^{m}K\cup
A^{m}KK\cup A^{m}KA^{m}K]=(KK]\cup (KA^{m}K]\cup (A^{m}KK]\cup
(A^{m}KA^{m}K]\subseteq (K]\cup (AA^{m}K]\cup (A^{m}AK]\cup
(A^{m}AA^{m}K]\subseteq (K]\cup (A^{m+1}K]\cup (A^{m+1}K]\cup
(A^{2m+1}K]\subseteq (K]\cup (A^{m}K]=(K\cup A^{m}K]$, using Lemma
\ref{AB}. Hence $(K)_{m-l}$ is a subsemigroup of $A$. Next we have
to show that $A^{m}(K)_{m-l}\subseteq (K)_{m-l}$. Suppose
$A^{m}(K)_{m-l}=A^{m}(K\cup A^{m}K]\subseteq (A^{m}K\cup
A^{2m}K]\subseteq (A^{m}K]\subseteq (K\cup A^{m}K]$. Hence $(K\cup
A^{m}K]$ is an m-left ideal of $A$. Now we need to show that
$(K)_{m-l}$ is the minimal m-left ideal of $A$ containing $K$.
Consider $K'$ be any m-left ideal of $A$ containing $K$. Now $(K\cup
A^{m}K]\subseteq (K'\cup A^{m}K']\subseteq (K'\cup K']=(K']=K'$.
Evidently $(K)_{m-l}\subseteq K'$. Hence $(K)_{m-l}$ is the minimal
m-left ideal containing $K$.

$(2)$: Similar As previous.

\end{proof}

\begin{corollary}
Let $A$ be an ordered semigroup. If $x\in A$, the m-left ideal
generated by $x$ denoted by $(x)_{m-l}$ and defined by
$(x)_{m-l}=(x\cup A^{m}x]$.
\end{corollary}

\begin{Definition}
Let $A$ be an ordered semigroup. An m-left ideal of $A$ is called
principal m-left ideal of $A$ if it is generated by a single element
of $A$.
\end{Definition}

\begin{theorem} In an ordered semigroup $A$, the following hold:
\begin{enumerate}
  \item \vspace{-.4cm}
For a subsemigroup $K$ of $A$, $(K)_{m-l}\subseteq (K)_{l}$.
  \item \vspace{-.4cm}
For any element $a\in A$, $(a)_{m-l}\subseteq (a)_{l}$.

\end{enumerate}
\end{theorem}
\begin{proof}
$(1)$: Since for any positive integer $m$, $(K\cup A^{m}K]\subseteq
(K\cup AK]$. Hence $(K)_{m-l}\subseteq (K)_{l}$.

$(2)$: Analogously.

\end{proof}

\begin{Definition}
Let m be non-negative integer. An ordered semigroup $(A,\cdot,\leq)$
is said to be m-left-simple(m-right-simple) if it does not contain
any  proper non trivial m-left(m-right) ideal.
\end{Definition}

\begin{lemma}
Let $A$ be an ordered semigroup and $m$ be any non-negative integer.
The following statements hold:

\begin{enumerate}
  \item \vspace{-.4cm}
$A$ is m-left-simple if and only if $A=(A^{m}x]$ for all $x\in A$.
  \item \vspace{-.4cm}
$A$ is m-right-simple if and only if $A=(xA^{m}]$ for all $x\in A$.

\end{enumerate}

\end{lemma}

\begin{proof}
$(1)$: Assume that $A$ is m-left simple and let $x\in A$. Now
$A^{m}(A^{m}x]\subseteq (A^{2m}x]\subseteq (A^{m}x]$. Hence
$(A^{m}x]$ is an m-left ideal of $A$. Hence $A=(A^{m}x]$, by
assumption.

Conversely, assume that $A=(A^{m}x]$ for all $x\in A$. Let $K$ be
any m-left ideal of $A$ then $A^{m}K\subseteq K$ and $(K]=K$. Let
$x\in K\subseteq A$ then $A=(A^{m}x]\subseteq (A^{m}K]\subseteq
(K]=K$. Hence $A$ is m-left simple.

$(2)$: This can be proved similarly.

\end{proof}

\begin{corollary}
Let $A$ be an ordered semigroup. The following statements hold for
an ordered semigroup:

\begin{enumerate}
  \item \vspace{-.4cm}
$A$ is left-simple if and only if $A=(As]$ for all $s\in A$.
  \item \vspace{-.4cm}
$A$ is right-simple if and only if $A=(sA]$ for all $s\in A$.

\end{enumerate}

\end{corollary}

\section{\textbf{m-bi-ideal  }}
\begin{Definition}
Let $A$ be an ordered semigroup and $B$ be subsemigroup of $A$, then
$B$ is called m-bi-ideal of $A$ if $BA^{m}B\subseteq B$ and $(B]=B$,
where $m\geq 1$ is a positive integer, called bipotency of bi-ideal
$B$.
\end{Definition}

\begin{theorem}
Every bi-ideal in an ordered semigroup $A$ is m-bi-ideal for any
$m\geq 1$.

\end{theorem}
\begin{proof}
Let $B$ be a bi-ideal of $A$, then by definition $BAB\subseteq B$
that is $BA^{1}B\subseteq B$ and $(B]=B$ which implies $B$ is an
m-bi-ideal of $A$.

\end{proof}
 But the converse is not true.

\begin{example}
Consider  an ordered semigroup $A=\{x,y,z,w\}$ with the
multiplication '$\cdot$' and the order relation defined by

\begin{tabular}{|c|c|c|c|c|c|}
  \hline

  . & x & y & z & w  \\ \hline
  x & x & x & x & x \\ \hline
  y & x & x & x & z  \\ \hline
  z & x & x & x & y  \\ \hline
  w & x & z & y & x \\

  \hline

\end{tabular}\\

$\leq =\{(x,x),(x,y),(y,y),(z,z),(w,w)\}$.

Then $\{x,w\}$ is an m-bi-ideal of $A$ for $m>1$ but not a bi-ideal
of $A$.

\end{example}

\begin{proposition}
Let $B_{1}, B_{2}$ be two m-bi-ideals with bipotencies $m_{1},m_{2}$
respectively of an ordered semigroup $A$. The product of any
$m_{1}$-bi-ideal and $m_{2}$-bi-ideal of $A$ is a
max$(m_{1},m_{2})$-bi-ideal of $S$. The product defined by
$B_{1}\ast B_{2}=(B_{1}B_{2}]$.
\end{proposition}

\begin{proof}
Let $B_{1}$ and $B_{2}$ be two m-bi-ideals of an ordered semigroup
$A$ with bipotencies $m_{1},m_{2}$ respectively then
$B_{1}A^{m_{1}}B_{1}\subseteq B_{1}$, $(B_{1}]=B_{1}$ and
$B_{2}A^{m_{2}}B_{2}\subseteq B_{2}$, $(B_{2}]=B_{2}$.
$(B_{1}B_{2}]$ is obviously a non-empty subset of $A$. Now
$(B_{1}B_{2}]^{2}\subseteq (B_{1}B_{2}](B_{1}B_{2}]\subseteq
(B_{1}B_{2}B_{1}B_{2}]\subseteq (B_{1}AB_{1}B_{2}]\subseteq
(B_{1}A\cdot1\cdot1\cdot1\cdot\cdot\cdot1B_{1}B_{2}]\subseteq
(B_{1}A^{m}B_{1}B_{2}]\subseteq (B_{1}B_{2}]$ thus $(B_{1}B_{2}]$ is
a subsemigroup of $A$. Now we have
$(B_{1}B_{2}]A^{max\{m_{1},m_{2}\}}(B_{1}B_{2}]\subseteq
(B_{1}B_{2}A^{max\{m_{1},m_{2}\}}(B_{1}B_{2}]\subseteq
(B_{1}A^{1+max\{m_{1},m_{2}\}}(B_{1}B_{2}]\subseteq
(B_{1}A^{m_{1}}B_{1}B_{2}]\subseteq (B_{1}B_{2}]$, using
\cite{Munir1}. Again $((B_{1}B_{2}]]=(B_{1}B_{2}]$. Hence $B_{1}\ast
B_{2}$ is an m-bi-ideal with bipotency max$\{m_{1},m_{2}\}$.

\end{proof}

\begin{proposition}
Let $A$ be an ordered semigroup and $X$ be an arbitrary subset of
$A$ and $B$ be an m-bi-ideal of $A$, $m$ is not necessarily $1$.
Then the $(BX]$ is also an m-bi-ideal of $A$.
\end{proposition}

\begin{proof}
Clearly $(BX]$ is a non-empty subset of $A$.Now $(BX]^{2}\subseteq
(BX](BX]\subseteq (BXBX]\subseteq (BABX]\subseteq
(BA\cdot1\cdot1\cdot\cdot\cdot1BX]\subseteq (BA^{m}BX]\subseteq
(BX]$. Now $(BX]A^{m}(BX]\subseteq (BXA^{m}BX]\subseteq
(BAA^{m}BX]\subseteq (BA^{1+m}BX]\subseteq (BA^{m}BX]\subseteq
(BX]$. Hence $(BX]$ is an m-bi-ideal of $A$ with bipotency $m$.

\end{proof}

Similarly we can show that $(XB]$ is an m-bi-ideal of $A$.

\begin{theorem}

The intersection of a family of m-bi-ideals of an ordered semigroup
$A$ with bipotencies $t_{1},t_{2},......$ is also an m-bi-ideal with
bipotency max$\{t_{1},t_{2},.....\}$. (Provided the intersection is
non-empty.)

\end{theorem}

\begin{proof}
Let $\{B_{\alpha}:\alpha\in \Delta\}$ be a family of m-bi-ideals of
an ordered semigroup $A$ with bipotencies $\{t_{\alpha}:\alpha\in
\Delta\}$. Now we have to show that $B=\cap_{\alpha\in
\Delta}B_{\alpha}$ is again an m-bi-ideal of $A$ with bipotency
max$\{t_{\alpha}:\alpha\in \Delta\}$. Now $B=\cap_{\alpha\in
\Delta}B_{\alpha}$ is also a subsemigroup of $A$ being intersection
of subsemigroups of $A$. Since each $B_{\alpha}$ is an m-bi-ideal
with bipotencies $t_{\alpha}$ for all $\alpha\in \Delta$ then
$B_{\alpha}A^{t_{\alpha}}B_{\alpha}\subseteq B_{\alpha}$ and
$(B_{\alpha}]=B_{\alpha}$. Now $BA^{max\{t_{\alpha}:\alpha\in
\Delta\}}B\subseteq B_{\alpha}A^{t_{\alpha}}B_{\alpha}\subseteq
B_{\alpha}$, for all $\alpha\in \Delta$. Again $(B]=(\cap_{\alpha\in
\Delta}B_{\alpha}]\subseteq \cap_{\alpha\in
\Delta}(B_{\alpha}]\subseteq \cap_{\alpha\in
\Delta}B_{\alpha}]\cap_{\alpha\in \Delta}B_{\alpha}=B$. Thus
$(B]=B$. Hence $\cap_{\alpha\in \Delta}B_{\alpha}$ is  an m-bi-ideal
of $A$ with bipotency max$\{t_{\alpha}:\alpha\in \Delta\}$.

\end{proof}

\section{\textbf{m-quasi-ideal  }}
\begin{Definition}
A subsemigroup $Q$ of an ordered semigroup $A$ is called
m-quasi-ideal of $A$ if $(A^{m}Q]\cap (QA^{m}]\subseteq Q$ and
$(Q]=Q$, where $m\geq 1$, called quasipotency of quasi-ideal $Q$.
\end{Definition}

\begin{theorem}

For any $m\geq 1$, a quasi-ideal is an m-quasi-ideal.

\end{theorem}
\begin{proof}
Let $Q$ be a quasi-ideal of $A$, then $(QA^{m}]\cap
(A^{m}Q]\subseteq
 (QA]\cap (AQ]\subseteq Q$ and $(Q]=Q$. Hence $Q$ is an m-quasi-ideal of
 $A$.
\end{proof}

\begin{corollary}

Every m-quasi-ideal of an ordered semigroup is an n-quasi-ideal for
all $n\geq m$.

\end{corollary}
\begin{proof}
Let $Q$ be an m-quasi-ideal of an ordered semigroup $A$ then
$(QA^{n}]\cap (A^{n}Q]\subseteq (QA^{m}]\cap (A^{m}Q]\subseteq Q$
that is $(QA^{n}]\cap (A^{n}Q]\subseteq Q$ and $(Q]=Q$. $Q$ is an
n-quasi-ideal of $A$.
\end{proof}

\begin{theorem}\label{LEFT-QUASI}

Every m-left/m-right ideal and hence every m-ideal of an ordered
semigroup $A$ is an m-quasi-ideal of $A$.

\end{theorem}
\begin{proof}
Let $L$ be an m-left ideal of an ordered semigroup $A$. Now
$(A^{m}L]\cap (LA^{m}]\subseteq (A^{m}L]\subseteq (L]=L$ infers that
$(A^{m}L]\cap (LA^{m}]\subseteq L$ and $(L]=L$. Hence $L$ is an
m-quasi-ideal of $A$.
\end{proof}

Hence From the previous example \ref{example1}, we can see that
$\{x,w\}$ is an m-quasi-ideal but not an quasi-ideal of $A$.

\begin{theorem}\label{intersection quasi}
The intersection of a family of m-quasi-ideals of an ordered
semigroup $A$ with quasipotencies $t_{1},t_{2},......$ is also an
m-quasi-ideal with quasipotency max$\{t_{1},t_{2},.....\}$.
(Provided the intersection is non-empty.)

\end{theorem}
\begin{proof}
Let $\{Q_{\lambda}:\lambda\in I\}$ be a family of m-quasi-ideals of
an ordered semigroup $A$ with quasipotencies $\{t_{\lambda}:
\lambda\in I\}$ then $(A^{t_{\lambda}}Q_{\lambda}]\cap
(Q_{\lambda}A^{t_{\lambda}}]\subseteq Q_{\lambda}$ and
$(Q_{\lambda}]=Q_{\lambda}$, for all $\lambda\in I$. Let
$t=max\{t_{\lambda}: \lambda\in I\}$ and $Q=\cap_{\lambda\in
I}Q_{\lambda}$ and $Q$ is a subsemigroup of $A$ being intersection
of subsemigroups of $A$. Now $(A^{t}Q]\cap (QA^{t}]\subseteq
(A^{t}Q_{\lambda}]\cap (Q_{\lambda}A^{t}]\subseteq
(A^{t_{\lambda}}Q_{\lambda}]\cap
(Q_{\lambda}A^{t_{\lambda}}]\subseteq Q_{\lambda}$, for all $\lambda
\in I$. Thus $(A^{t}Q_{\lambda}]\cap (Q_{\lambda}A^{t}]\subseteq Q$
and $(Q]=(\cap_{\lambda\in I}Q_{\lambda}]\subseteq \cap_{\lambda\in
I}(Q_{\lambda}]=\cap_{\lambda\in I}Q_{\lambda}=Q$. Hence
$\cap_{\lambda\in I}Q_{\lambda}$ is an m-quasi-ideal with
quasi-potency max$\{t_{\lambda}: \lambda\in I\}$.
\end{proof}

\begin{lemma}\label{LR}
Let $A$ be an ordered semigroup and $L$ be an m-left ideal and $R$
be an m-right ideal of $A$. Then $R\cap L$ is an m-quasi-ideal of
$A$.

\end{lemma}
\begin{proof}
Since $RL\subseteq RS\subseteq
R\cdot1\cdot1\cdot1\cdot\cdot\cdot1\cdot S\subseteq RS^{m}\subseteq
R$ and similarly $RL\subseteq L$ infers that $RL\subseteq R\cap L$
thus $R\cap L\neq \emptyset$. Now $(R\cap L]\subseteq (R]\cap
(L]\subseteq R\cap L$ hence $(R\cap L]=R\cap L$. Then by Theorem
\ref{LEFT-QUASI},\ref{intersection quasi} follows that $R\cap L$ is
an m-quasi-ideal of $A$.
\end{proof}

\begin{Definition}
A subsemigroup $P$ of an ordered semigroup $A$ has the
m-intersection property if $P$ is  intersection of an m-left ideal
and an m-right ideal of $A$.
\end{Definition}

\begin{theorem}\label{intersection property}
An m-quasi-ideal $Q$ of an ordered semigroup $A$ has the
m-intersection property if and only if $(Q\cup A^{m}Q]\cap (Q\cup
QA^{m}]=Q$.

\end{theorem}
\begin{proof}
First consider $Q$ has the m-intersection property. Its evident that
$Q\subseteq (Q\cup A^{m}Q]$ and $Q\subseteq (Q\cup QA^{m}]$ infers
that $Q\subseteq (Q\cup A^{m}Q]\cap (Q\cap QA^{m}]$. Since $Q$ has
the m-intersection property then there exists an m-left ideal $L$
and m-right ideal $R$ such that $Q=L\cap R$. Thus $Q\subseteq L$ and
$Q\subseteq R$ so that $(A^{m}Q]\subseteq (A^{m}L]\subseteq (L]=L$
and $(QA^{m}]\subseteq (RA^{m}]\subseteq (R]=R$. Now $(Q\cup
A^{m}Q]\subseteq Q\cup (A^{m}Q]\subseteq L$ and $(Q\cup
QA^{m}]\subseteq Q\cup (QA^{m}]\subseteq R$. Hence $(Q\cup
A^{m}Q]\cap (Q\cup QA^{m}]\subseteq L\cap R=Q$. Hence $Q=(Q\cup
A^{m}Q]\cap (Q\cup QA^{m}]$.

Conversely, Suppose $Q=(Q\cup A^{m}Q]\cap (Q\cup QA^{m}]$. Since it
is clear that from Theorem \ref{Principal left ideal} $(Q\cup
A^{m}Q]$ and $(Q\cup QA^{m}]$ are m-left and m-right ideal
respectively. Hence $Q$ has the m-intersection property.

\end{proof}

\begin{lemma}\label{Q=L(Q)R(Q)}
Let $Q$ be an m-quasi-ideal of an ordered semigroup $A$ then
$Q=L(Q)\cap R(Q)=(Q\cup A^{m}Q]\cap (Q\cup QA^{m}]$.
\end{lemma}
\begin{proof}
The inclusion $Q\subseteq (Q\cup A^{m}Q]\cap (Q\cup QA^{m}]$ is
always true.

Conversely, let $a\in (A^{m}Q\cup Q]\cap (Q\cup QA^{m}]$ then $a\leq
w$ or $a\leq xb$ and $a\leq cy$ for some $x,y\in A^{m}$, $w,b,c\in
Q$. Since $Q$ is m-quasi-ideal then $a\leq w$ infers that $a\in
(Q]=Q$. Again $a\leq xb$ and $a\leq cy$ implies that $a\in
(A^{m}Q]\cap (QA^{m}]\subseteq Q$. Thus $(Q\cup A^{m}Q]\cap (Q\cup
QA^{m}]=Q$.

\end{proof}

Let $B$ be a non-empty subset of an ordered semigroup $A$. We denote
the least m-quasi-ideal of containing $B$ by $Q(B)$. If  $B={b}$, we
denote  $Q(\{b\})$ by $Q(b)$.
\begin{corollary}
Let $A$ be an ordered semigroup. Then
\begin{enumerate}

  \item \vspace{-.4cm}
For every $x\in A$, $Q(x)=L(x)\cap R(x)=(x\cup A^{m}x]\cap (x\cup
xA^{m}]$
 \item \vspace{-.4cm}
For every  $\emptyset\neq X\subseteq A$, $Q(X)=L(X)\cap R(X)=(X\cup
A^{m}X]\cap (X\cup XA^{m}]$.
\end{enumerate}
\end{corollary}

\begin{proof}
$(1)$: Let $x\in A$, Since $x\in L(x)\cap R(x)$, hence non-empty
then by \ref{LR}, $L(x)\cap R(x)$ is an m-quasi-ideal infers that
$Q(x)\subseteq L(x)\cap R(x)$. Again by Theorem \ref{Principal left
ideal}, $L(x)\cap R(x)=(x\cup A^{m}x]\cap (x\cup xA^{m}]\subseteq
(Q(x)\cup A^{m}Q(x)]\cap (Q(x)\cup Q(x)A^{m}]=Q(x)$. Thus
$Q(x)=L(x)\cap R(x)$.

$(2)$: It can be proved easily.
\end{proof}

It is simple to verify this every m-left and m-right ideal of an
ordered semigroup $A$ have the intersection property.

Let $L$ be an m-left ideal of an ordered semigroup $A$. By Theorem
\ref{LEFT-QUASI} its obvious that $L$ is an m-quasi ideal of $A$.
Now $(L\cup A^{m}L]\cap (L\cup LA^{m}]=\{(L]\cup (A^{m}L]\}\cap
\{(L]\cup (LA^{m}]\}=\{L\cup (A^{m}L]\}\cap \{L\cup (LA^{m}]\}=L\cup
((LA^{m}]\cap (A^{m}L])\subseteq L\cup L=L$. Again $L\subseteq L\cup
((A^{m}L]\cap (LA^{m}])=\{L\cup (A^{m}L]\}\cap \{L\cup
(LA^{m}]\}=\{(L]\cup (A^{m}L]\}\cap \{(L]\cup (LA^{m}]\}=(L\cup
A^{m}L]\cap (L\cup LA^{m}]$. Hence by previous result
\ref{intersection property} $L$ has the m-intersection property.
Likewise we can show that every m-right ideal has the m-intersection
property.

\begin{theorem}
For m-quasi-ideal $Q$ of an ordered semigroup $A$ if
$A^{m}Q\subseteq QA^{m}$ or $QA^{m}\subseteq A^{m}Q$ then $Q$ has
the m-intersection property.

\end{theorem}
\begin{proof}
Without loss of generality, consider that $A^{m}Q\subseteq QA^{m}$
then $A^{m}Q=A^{m}Q\cap QA^{m}\subseteq (A^{m}Q]\cap
(QA^{m}]\subseteq Q$ infers that $A^{m}Q\subseteq Q$ that is $Q$ is
an m-left ideal of $A$. Thus $Q$ has the m-intersection property.
\end{proof}

\begin{theorem}\label{bi-quasi}
Every m-quasi-ideal of an ordered semigroup $A$ is its m-bi-ideal.

\end{theorem}
\begin{proof}
Suppose that $Q$ be an m-quasi-ideal of an ordered semigroup $A$
then $QA^{m}Q\subseteq QA^{m}A\cap AA^{m}Q\subseteq QA^{m+1}\cap
A^{m+1}Q\subseteq QA^{m}\cap A^{m}Q\subseteq (QA^{m}]\cap
(A^{m}Q]\subseteq Q$ and $(Q]=Q$. $Q$ is an m-bi-ideal of $A$.
\end{proof}
\begin{Definition}
An element $a$ of an ordered semigroup $A$ is called m-regular if
$a\leq axa$, for some $x\in A^{m}$. An ordered semigroup $A$ is
m-regular if every element of $A$ is m-regular that is m-regular if
$a\in (aA^{m}a]$, for all $a\in A$.
\end{Definition}
Every regular ordered semigroup is an m-regular ordered semigroup
but the converse is not true.

\begin{lemma}\label{ASm}
Let $A$ be an m-regular ordered semigroup and $T$ be any non-empty
subset of $A$. Then
\begin{enumerate}

 \item \vspace{-.4cm}
$X\subseteq (TA^{m}]$, for any positive integer $m$.
 \item \vspace{-.4cm}
$X\subseteq (A^{m}T]$, for any positive integer $m$.

\end{enumerate}
\end{lemma}
\begin{proof}
$(1)$: Let $a\in T$ then $a\leq axa$ for some $x\in A^{m}$, Since
$A$ is m-regular. Hence $a\in (aA^{m}a]\subseteq (TA^{m}T]\subseteq
(TA^{m}A]\subseteq (TA^{m+1}]\subseteq (TA^{m}]$ infers that
$T\subseteq (TA^{m}]$.

$(2)$: Similarly we can prove that $T\subseteq (A^{m}T]$.

\end{proof}

\begin{theorem}\label{m-regular}
The following conditions for an ordered semigroup $A$ are
equivalent:
\begin{enumerate}

 \item \vspace{-.4cm}
$A$ is m-regular.
 \item \vspace{-.4cm}
For every m-right ideal $R$ and m-left ideal $L$ of $A$, $(RL]=R\cap
L$.
 \item \vspace{-.4cm}
For every m-right ideal $R$ and m-left ideal $L$ of $A$,
\begin{enumerate}

\item \vspace{-.4cm}
$(R^{2}]=R$.
\item \vspace{-.4cm}
$(L^{2}]=L$
\item \vspace{-.4cm}
$(RL]$ is a m-quasi-ideal of $A$.

\end{enumerate}

\item \vspace{-.4cm}
The set of m-quasi ideals $Q_{A}$ of $A$ with the multiplication
$\circ$ defined on it by $Q_{1}\circ Q_{2}=(Q_{1}Q_{2}]$, for all
$Q_{1},Q_{2}\in Q_{A}$, is a m-regular semigroup.

\item \vspace{-.4cm}
Every m-quasi-ideal $Q$ has the form $(QA^{m}Q]=Q$.

\end{enumerate}

\end{theorem}
\begin{proof}
$(1)\Rightarrow (2)$: First consider that $A$ is an m-regular
ordered semigroup. Let $R$ and $L$ be m-right ideal and m-left ideal
of $A$ respectively. Now by Lemma \ref{ASm}, $(RL]\subseteq
(A^{m}RL]\subseteq (A^{m}AL]\subseteq (A^{m+1}L]\subseteq
(A^{m}L]\subseteq (L]=L$. Similarly $(RL]\subseteq R$. Hence
$(RL]\subseteq R\cap L$. Again let $x\in R\cap L$. Since $A$ is
m-regular then $x\leq xyx$ for some $y\in A^{m}$ which implies that
$x\in (xA^{m}x]\subseteq (RA^{m}L]\subseteq (RL]$. Therefore $R\cap
L=(RL]$.

$(2)\Rightarrow (3)$: First assume that $(RL]=R\cap L$. Now by Lemma
\ref{LR}, $(RL]$ is an m-quasi-ideal of $A$. Now for ordered
semigroup $A$, then according to Theorem \ref{Principal left ideal}
m-left ideal generated by $R$ is $(R\cup A^{m}R]$. Now $R=R\cap
(R\cup A^{m}R]=(R(R\cup A^{m}R]]$ which infers that
$(R^{2}]\subseteq (R(R\cup A^{m}R]]=R$. Conversely, let $p\in
(R(R\cup A^{m}R]]$. Then $p\leq r_{1}x$, for some $r_{1}\in R$ and
$x\in (R\cup A^{m}R]$. Now $x\in (R\cup A^{m}R]=(R]\cup
(A^{m}R]=R\cup (A^{m}R]$ which implies that $x\leq r_{2}$ for some
$r_{2}\in R$ or $x\leq sr_{3}$ where $s\in A^{m}$, $r_{3}\in R$.
Hence $p\leq r_{1}sr_{3}\in (RA^{m}R]\subseteq (R^{2}]$. So $p\in
(R^{2}]$. Thus $R\subseteq (R^{2}]$ that is $(R^{2}]=R$.

Similarly we can show that $(L^{2}]=L$.

$(3)\Rightarrow (4)$: Suppose $Q_{A}$ be the set of all
m-quasi-ideals of $A$. Now for any $Q\in Q_{A}$, $(Q\cup A^{m}Q]$ is
the m-left ideal of $A$ generated by $Q$. Now by condition $3(b)$,
we have $Q\subseteq (Q\cup A^{m}Q]=((Q\cup A^{m}Q]^{2}]\subseteq
((Q\cup A^{m}Q](Q\cup A^{m}Q]]\subseteq (Q^{2}\cup A^{m}Q^{2}\cup
QA^{m}Q\cup (A^{m}Q)^{2}]=(Q^{2}]\cup (A^{m}QQ]\cup (QA^{m}Q]\cup
((A^{m}Q)(A^{m}Q)]\subseteq (Q^{2}]\cup (A^{m}AQ]\cup (AA^{m}Q]\cup
((A^{m}AA^{m}Q)]\subseteq (Q^{2}]\cup (A^{m}Q]\subseteq
(A\cdot1\cdot1\cdot1\cdot\cdot\cdot\cdot1\cdot Q]\cup
(A^{m}Q]\subseteq (A^{m}Q]$. Similarly we can prove that $Q\subseteq
(QA^{m}]$. Hence $Q\subseteq (A^{m}Q]\cap (QA^{m}]$. Since $Q$ is an
m-quasi-ideal of $A$ then $(A^{m}Q]\cap (QA^{m}]\subseteq Q$ which
follows that $Q=(A^{m}Q]\cap
(QA^{m}]$.$\cdot\cdot\cdot\cdot\cdot\cdot\cdot\cdot\cdot\cdot$ $(i)$

Now, from condition $(3)(c)$, it follows that $(RL]=((RL]A^{m}]\cap
(A^{m}(RL]]$.
$\cdot\cdot\cdot\cdot\cdot\cdot\cdot\cdot\cdot\cdot\cdot\cdot$
$(ii)$ for every m-left ideal $L$ and m-right ideal $R$ of $A$. Now
we have to prove that the product $Q_{1}\circ Q_{2}=(Q_{1}Q_{2}]$ of
two m-quasi-ideals $Q_{1}$ and $Q_{2}$ is also an m-quasi ideal of
$S$. $(Q_{1}Q_{2}]^{2}=(Q_{1}Q_{2}](Q_{1}Q_{2}]\subseteq
(Q_{1}Q_{2}Q_{1}Q_{2}]\subseteq (Q_{1}AQ_{1}Q_{2}]\subseteq
(Q_{1}A\cdot 1\cdot 1\cdot\cdot\cdot\cdot\cdot\cdot\cdot
Q_{1}Q_{2}]\subseteq (Q_{1}\cdot A\cdot A\cdot A\cdot\cdot\cdot
AQ_{1}Q_{2}]\subseteq ((Q_{1}A^{m}Q_{1})Q_{2}]\subseteq
(Q_{1}Q_{2}]$, by Theorem \ref{bi-quasi}. Hence
$(Q_{1}Q_{2}]^{2}=(Q_{1}Q_{2}]$ that is $(Q_{1}Q_{2}]$ is a
subsemigroup of $A$. Now $(A^{m}Q_{1}Q_{2}]$ is an m-left ideal.
Using result $(3)(a)$ and $(3)(b)$,
$(A^{m}Q_{1}Q_{2}]=((A^{m}Q_{1}Q_{2}](A^{m}Q_{1}Q_{2}]]
\subseteq((A^{m}Q_{1}Q_{2}](A^{m}Q_{1}Q_{2}]^{2}]
\subseteq(A^{m}Q_{1}Q_{2}A^{m}Q_{1}Q_{2}A^{m}Q_{1}Q_{2}] \subseteq
(A^{m}Q_{1}Q_{2}A^{m}AAA^{m}Q_{1}Q_{2}]\subseteq
(A^{m}Q_{1}Q_{2}A^{2m+2}\\Q_{1}Q_{2}] \subseteq
(A^{m}((Q_{1}Q_{2}A^{m}](A^{m}Q_{1}Q_{2}]]]$.

Similarly, $(Q_{1}Q_{2}A^{m}]\subseteq
(((Q_{1}Q_{2}A^{m}](Q_{1}Q_{2}A^{m}]]A^{m}]$.

Now $(A^{m}(Q_{1}Q_{2}]]\cap ((Q_{1}Q_{2}]A^{m}]\subseteq
(A^{m}Q_{1}Q_{2}]\cap (Q_{1}Q_{2}A^{m}]= (Q_{1}Q_{2}A^{m}]\cap
(A^{m}Q_{1}Q_{2}]\subseteq
(((Q_{1}Q_{2}A^{m}](Q_{1}Q_{2}A^{m}]]A^{m}]\cap
(A^{m}((Q_{1}Q_{2}A^{m}](A^{m}Q_{1}Q_{2}]]]\subseteq
((Q_{1}Q_{2}A^{m}](A^{m}Q_{1}Q_{2}]]\subseteq
(Q_{1}Q_{2}A^{m}A^{m}Q_{1}Q_{2}]\subseteq
(Q_{1}AA_{m}A^{m}Q_{1}Q_{2}]\subseteq
(Q_{1}A^{2m+1}Q_{1}Q_{2}]\subseteq (Q_{1}A^{m}Q_{1}Q_{2}]\subseteq
((Q_{1}A^{m}Q_{1}]Q_{2}]\subseteq ((Q_{1}]Q_{2}]\subseteq
(Q_{1}Q_{2}]$.

Hence $(Q_{1}Q_{2}]$ is an m-quasi-ideal of $A$. Since the
multiplication defined in $Q_{A}$ is associative, so $(Q_{A},\circ)$
is a semigroup. Now we have to show that $(Q_{A},\circ)$ is regular.
Using the conditions $(3)(a)$ and $(3)(b)$, we have $A^{2}=A$ and

$(A^{m}Q]=(A^{m}Q]^{2}=((A^{m}Q](A^{m}Q]]\subseteq
((A^{m}Q](A^{2m}Q]]\subseteq (A^{m}QA^{2m}Q]\subseteq
(A^{m}(QA^{m}](A^{m}Q]$. Similarly $(QA^{m}]\subseteq
((QA^{m}](A^{m}Q]A^{m}]$. therefore from relations $(i),(ii)$ it
follows that $Q=(A^{m}Q]\cap (QA^{m}]\subseteq
(A^{m}(QA^{m}]\\(A^{m}Q]\cap ((QA^{m}](A^{m}Q]A^{m}]\subseteq
((QA^{m}](A^{m}Q]]\subseteq (QA^{m}Q]\subseteq (Q]=Q$. Hence
$Q=(QA^{m}Q]=Q\circ A^{m}\circ Q$. This means that $(Q_{A},\circ)$
is a regular semigroup.

$(4)\Rightarrow (5)$: Let $Q$ be an m-quasi-ideal of $A$. By the
assumption $(4)$, there exists an m-quasi-ideal $P$ of $A$ such that
$Q=Q\circ P^{m}\circ Q=(QP^{m}Q]\subseteq (QA^{m}Q]\subseteq
(QA^{m}]\cap (A^{m}Q]\subseteq Q$ which infers that $Q=(QA^{m}Q]$.

$(5)\Rightarrow (1)$: Let $x\in A$ and $L(x)$ and $R(x)$ be the
principal m-left and principal m-right ideal of $A$ generated by
$x$. Now Lemma \ref{LR} infers that $R(x)\cap L(x)$ is an
m-quasi-ideal of $A$. Then by condition $(5)$, we have $x\in
R(x)\cap L(x)=((R(x)\cap L(x))A^{m}(R(x)\cap L(x))]\subseteq
(R(x)A^{m}L(x)]\subseteq ((x\cup xA^{m}]A^{m}(x\cup
A^{m}x]]\subseteq (xA^{m}x]$. Thus $A$ is m-regular.

\end{proof}

\begin{lemma}\label{regular ideal}
Every two sided m-ideal $J$ of an m-regular ordered semigroup $A$ is
an m-regular subsemigroup of $A$.
\end{lemma}

\begin{proof}
By Theorem \ref{LEFT-QUASI}, $J$ is an m-quasi-ideal of $A$. Let
$a\in J$. Since $A$ is m-regular, then there exists $x\in A^{m}$
such that $a\leq axa\leq a(xax)a$. Since $xax\in
A^{m}JA^{m}\subseteq J$. Hence we have $a\in (aJa]$. Thus $a$ is
regular in $J$.
\end{proof}

\begin{theorem}
Let $A$ be an m-regular ordered semigroup, then the following
assertions hold:
\begin{enumerate}

\item \vspace{-.4cm}
Every m-quasi ideal $Q$ of $A$ can be written in the form $Q=R\cap
L=(RL]$. where $R$ is m-right ideal and $L$ is m-left ideal of $A$.
\item \vspace{-.4cm}
If $Q$ is m-quasi-ideal of $A$, then $(Q^{2}]=(Q^{3}]$.

\item \vspace{-.4cm}
Every m-bi-ideal of $A$ is a m-quasi-ideal of $A$.

\item \vspace{-.4cm}
Every m-bi-ideal of any two-sided ideal of $A$ is an m-quasi-ideal
of $A$.

\end{enumerate}

\end{theorem}
\begin{proof}
$(1)$: Since $Q$ is an m-quasi ideal of an m-regular ordered
semigroup $A$. Now by Lemma \ref{Q=L(Q)R(Q)} and Theorem
\ref{m-regular}, it follows that $Q=R\cap L=(RL]$.

$(2)$: $(Q^{3}]\subseteq (Q^{2}]$ always hold, we have to show that
$(Q^{2}]\subseteq (Q^{3}]$. By previous theorem \ref{m-regular},
$(Q^{2}]$ is an m-quasi-ideal of $A$ and in-addition,
$(Q^{2}]=(Q^{2}A^{m}Q^{2}]=(QQA^{m}QQ]=(Q(QA^{m}Q)Q]\subseteq
(QQQ]=(Q^{3}]$ which implies that $(Q^{2}]\subseteq (Q^{3}]$. Hence
$(Q^{2}]=(Q^{3}]$.

$(3)$: Let $B$ be an m-bi-ideal of $A$ then $(A^{m}B]$ is an m-left
ideal and $(BA^{m}]$ is an m-right ideal of $A$. By Theorem
\ref{m-regular}, $(BA^{m}]\cap (A^{m}B] =(BA^{m}A^{m}B]\subseteq
(BA^{2m}B]\subseteq (BA^{m}B]\subseteq (B]=B$. Hence $B$ is an
m-quasi-ideal of $A$.

$(4)$: Let $I$ be two sided m-ideal of $A$, and $B$ be an m-bi-ideal
of $I$. According to Lemma \ref{regular ideal}, $I$ is an m-regular
subsemigroup of $A$. By previous property $(3)$ $B$ is m-quasi-ideal
of $I$. Now $BA^{m}B\subseteq BA^{m}I$ and $BA^{m}B\subseteq
IA^{m}B$, so $BA^{m}B\subseteq BA^{m}I\cap IA^{m}B\subseteq BI\cap
IB\subseteq (BI]\cap (IB]\subseteq B$ that is $BA^{m}B\subseteq B$.
Hence $B$ is an m-bi-ideal of $A$. Therefore by previous condition
$(3)$, $B$ is an m-quasi-ideal of $A$.

\end{proof}

\bibliographystyle{amsplain}

\begin{thebibliography}{10}
\baselineskip 5mm
\bibitem{Munir1}
M. Munir and M.Habib. \emph{Characterizing semirings using their
quasi and bi-ideals}. Proc. Pak- istan Acad. Sci.,
\textbf{53}(4)(2016), 469-475.

\bibitem{K1}
N. Kehayopulu,\emph{ On weakly prime ideals of ordered semigroups},
Mathematica Japonica \textbf{35} No.6(1990), 1051-1056.

\bibitem{K2}

N. Kehayopulu, \emph{Note on Green's relation in ordered
semigroups}, Mathematica Japonica \textbf{36} No.2(1991), 211-214.

\bibitem{kehayopulu2}
N. Kehayopulu and M. Tsingelis, \emph{On left regular ordered
semigroups}, Southeast Asian Bull. Math. \textbf{25}(2002), 609-615.

\bibitem{H1}
K. Hansda, \emph{Minimal bi-ideals in regular and completely regular
ordered semigroups}, Quasigroup and related systems, Vol.
\textbf{27}(1)(2019), 63-72.



\end{thebibliography}

\end{document}